\documentclass[12pt, reqno]{amsart}

\usepackage{amssymb}
\usepackage{amsmath}
\usepackage{amsthm}
\usepackage{enumerate}
\usepackage[usenames]{color}
\usepackage{graphicx}
\usepackage{breqn}
\usepackage[normalem]{ulem}
\usepackage[left=3cm,top=3cm,right=3cm,bottom=3cm]{geometry}
\usepackage{tikz}

\usepackage[usenames]{color}
\usepackage{mathrsfs}

\def\E{\mathbb{E}}

\title[On offset Hamilton cycles in random hypergraphs]{On offset Hamilton cycles in random hypergraphs}

\author{Andrzej Dudek and Laars Helenius}
\address[Andrzej Dudek and Laars Helenius]{Department of Mathematics, Western Michigan University, Kalamazoo, MI}
\email{\tt \{andrzej.dudek,\;laars.c.helenius\}@wmich.edu}
\thanks{The first author was supported in part by Simons Foundation Grant \#244712 and by the National Security Agency under Grant Number H98230-15-1-0172. The United States Government is authorized to reproduce and distribute reprints notwithstanding any copyright notation hereon.}

\newtheoremstyle{plain}%
{8pt plus2pt minus4pt}%
{8pt plus2pt minus4pt}%
{\itshape}%
{}%
{\bfseries\scshape}%
{}%
{6pt}
{}%

\newtheoremstyle{remark}%
{8pt plus2pt minus4pt}%
{8pt plus2pt minus4pt}%
{\upshape}
{}%
{\bfseries\scshape}%
{}%
{6pt}
{}%

\theoremstyle{plain}
\newtheorem{theorem}{Theorem}

\newtheorem{observation}{Observation}

\theoremstyle{remark}

\begin{document} 

\begin{abstract}
An {\em $\ell$-offset Hamilton cycle} $C$ in a $k$-uniform hypergraph $H$ on~$n$ vertices is a collection of edges of $H$ such that for some cyclic order of $[n]$ every pair of consecutive edges $E_{i-1},E_i$ in $C$ (in the natural ordering of the edges) satisfies $|E_{i-1}\cap E_i|=\ell$ and every pair of consecutive edges $E_{i},E_{i+1}$ in $C$ satisfies $|E_{i}\cap E_{i+1}|=k-\ell$. We show that in general $\sqrt{e^{k}\ell!(k-\ell)!/n^k}$ is the sharp threshold for the existence of the $\ell$-offset Hamilton cycle in the random $k$-uniform hypergraph $H_{n,p}^{(k)}$. We also examine this structure's natural connection to the 1-2-3 Conjecture.
\end{abstract}

\maketitle

\section{Introduction}

A $k$-uniform hypergraph is a hypergraph in which each edge contains exactly $k$ vertices. The random $k$-uniform hypergraph, denoted $H_{n,p}^{(k)}$, has each possible edge appearing independently with probability $p$. Observe that $H_{n,p}^{(2)}$ is equivalent to the binomial random graph $G_{n,p}$. 

The threshold for the existence of Hamilton cycles in the random graph $G_{n,p}$ has been known for many years, see, e.g., \cite{AKS}, \cite{Boll} and \cite{KS}. There have been many generalizations of these results over the years and the problem is well understood. Quite recently some of these results were extended to hypergraphs.

Suppose that $1\leq \ell< k$. An {\em $\ell$-overlapping Hamilton cycle} $C$
in a $k$-uniform hypergraph $H=(V,\mathcal{E})$ on $n$ vertices is a
collection of $m_\ell=n/(k-\ell)$ edges of $H$ such that for some cyclic order
of $[n]$ every edge consists of $k$ consecutive vertices and for
every pair of consecutive edges $E_{i-1},E_i$ in $C$ (in the natural
ordering of the edges) we have $|E_{i-1}\cap E_i|=\ell$. Thus, in every
$\ell$-overlapping Hamilton cycle the sets $C_i=E_i\setminus
E_{i-1},\,i=1,2,\ldots,m_\ell$, are a partition of $V$ into sets of
size $k-\ell$. Hence, $m_{\ell}=n/(k-\ell)$. Thus, $k-\ell$ divides $n$.
In the literature, when $\ell=k-1$ we have a {\em tight} Hamilton cycle and
when $\ell=1$ we have a {\em loose} Hamilton cycle.

A $k$-uniform hypergraph is said to be {\em $\ell$-Hamiltonian} when it contains an $\ell$-overlapping Hamilton cycle.
Recently, results on loose hamiltonicity of $H_{n,p}^{(k)}$ were
obtained by Frieze~\cite{F} (for $k=3$), Dudek and Frieze~\cite{DF} (for $k \ge 4$ and $2(k-1)|n$),
and by Dudek, Frieze, Loh and Speiss~\cite{DFLS} (for $k \ge 3$ and $(k-1)|n$).

Throughout this paper the following conventions are adhered to: we let $\omega=\omega(n)$ be any function tending to infinity with $n$, we let $e$ be the base of natural logarithm $\log n$, and we do not round numbers that are supposed to be integers either up or down. This last convention is justified since the rounding errors introduced are negligible for the asymptomatic calculations we make.

\begin{theorem}[\cite{F, DF,DFLS}]\label{thm:loose}
There exists an absolute constant $c>0$ such that if $p\geq c(\log n)/n^2$, then a.a.s. $H_{n,p}^{(3)}$ contains a loose Hamilton cycle provided that $2|n$. Furthermore, for every $k \ge 4$ if $p \ge \omega(\log n)/n^{k-1}$, then $H_{n,p}^{(k)}$ contains a loose Hamilton cycle provided that $(k-1)|n$.
\end{theorem}
\noindent
These results are basically optimal since if $p\leq (1-\varepsilon) (k-1)! (\log n) / n^{k-1}$ and $\varepsilon>0$ is constant, then a.a.s. $H_{n,p}^{(k)}$ contains isolated vertices. More recently Ferber~\cite{Feb} simplified some of the proofs of Theorem~\ref{thm:loose} and Dudek and Frieze~\cite{DF2} were able to extend these to an arbitrary $\ell\ge 2$.
\begin{theorem}[\cite{DF2}]\label{thm:tight}
\ \\[-0.1in]
\begin{enumerate}[(i)]
\item For all integers $k> \ell\geq 2$ and fixed $\varepsilon>0$, if
$p\leq (1-\varepsilon)e^{k-\ell}/n^{k-\ell}$, then
a.a.s.\ $H_{n,p}^{(k)}$ is not $\ell$-Hamiltonian.
\item For all integers $k>\ell \ge 3$, there exists a
constant $c=c(k)$ such that if $p\geq c/n^{k-\ell}$
and $n$ is a multiple of $k-\ell$, then a.a.s.
 $H_{n,p}^{(k)}$ is $\ell$-Hamiltonian. \label{thm:a}
\item If $k>\ell=2$ and $p\geq \omega/n^{k-2}$
and $n$ is a multiple of $k-2$,
then a.a.s. $H_{n,p}^{(k)}$ is $2$-Hamiltonian.\label{thm:b}
\item For a fixed $\varepsilon>0$, if $k\geq 4$ and $p\geq (1+\varepsilon)e/n$, then
a.a.s.\ $H_{n,p}^{(k)}$ contains a tight Hamilton cycle.
\label{thm:c}
\end{enumerate}
\end{theorem}
\noindent
This theorem shows, in particular, that $e/n$ is the sharp threshold for the existence of a tight Hamilton cycle in a $k$-uniform hypergraph, when
$k\geq 4$.

Finally Poole~\cite{Poole} considered \emph{weak (Berge) Hamiltonian cycles} $C$ in $k$-uniform hypergraphs $H$ on $n$ vertices which are collections of edges of $H$ such that for some cyclic order
of $[n]$ every pair of consecutive vertices belong to an edge from $C$ and these edges are not necessarily distinct. Notice that loose Hamilton cycles are weak Hamiltonian cycles, too. In particular,
\begin{theorem}[\cite{Poole}]\label{thm:berge}
Let $k\ge 3$. Then, $p=(k-1)! (\log n)/ n^{k-1}$ is a sharp threshold for the existence of the weak Hamiltonian cycle in $H_{n,p}^{(k)}$.
\end{theorem}

In this paper we let  $1\leq \ell\le k/2$ and define an {\em $\ell$-offset Hamilton cycle} $C$ in a $k$-uniform hypergraph $H$ on $n$ vertices as a collection of $m$ edges of $H$ such that for some cyclic order of $[n]$ every pair of consecutive edges $E_{i-1},E_i$ in $C$ (in the natural ordering of the edges) satisfies $|E_{i-1}\cap E_i|=\ell$ and every pair of consecutive edges $E_{i},E_{i+1}$ in $C$ satisfies $|E_{i}\cap E_{i+1}|=k-\ell$~(see Figure~\ref{fig:offsetHcycle}).

Since every $\ell$-offset Hamilton cycle consists of two perfect matching of size $n/k$, we have $m=2n/k$ and we always assume that $k$ divides $n$ when discussing $\ell$-offset Hamilton cycles. A $k$-uniform hypergraph is said to be {\em $\ell$-offset Hamiltonian} when it contains an $\ell$-offset Hamilton cycle. 

\begin{figure}[h]
\begin{tikzpicture}
\foreach \m in {0,1,2}{
\draw [fill=gray!50!white, rounded corners, rotate around={120*\m:(0,0)}] 
    ( 15:2.25) arc ( 15:129:2.25) -- 
    (129:1.75) arc (129: 15:1.75) -- cycle;
}
\foreach \m in {0,1,2}{
\draw [rounded corners, rotate around={120*\m:(0,0)}]
     ( 63:1.65) arc ( 63:177:1.65) --
     (177:2.35) arc (177: 63:2.35) -- cycle;
}
\foreach \n in {1,...,15}{\draw[fill=black] (\n*24:2) circle (0.05);};
    \end{tikzpicture}
\caption{A $2$-offset Hamilton cycle in a $5$-uniform hypergraph.}
\label{fig:offsetHcycle}
\end{figure}
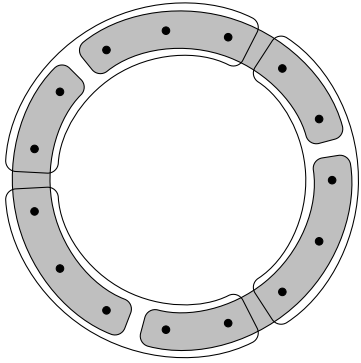

It follows from a result of Parczyk and Person (Corollary 3.1 in~\cite{PP}) that if $p=\omega/n^{k/2}$, then a.a.s. $H_{n,p}^{(k)}$ is $\ell$-offset Hamiltonian for any $k\ge 4$ and $\ell \ge 2$. In the next theorem we replace this asymptotic threshold by the sharp one for $k\ge 6$ and $\ell\ge 3$. 

\newpage

\begin{theorem}\label{thm:main} 
Let $\varepsilon >0$. Then:
\begin{enumerate}[(i)]
\item For all integers $k\ge 3$ and $1\le\ell\le\frac k2$, if $p=(1-\varepsilon)\sqrt{e^{k}\ell!(k-\ell)!/n^{k}}$,
then a.a.s.\ $H_{n,p}^{(k)}$ is not $\ell$-offset Hamiltonian.\label{thm:a}
\item For all integers $k\ge 6$ and $3\le\ell\le\frac k2$, if $p=(1+\varepsilon)\sqrt{e^{k}\ell!(k-\ell)!/n^{k}}$,
then a.a.s.\ $H_{n,p}^{(k)}$ is $\ell$-offset Hamiltonian.\label{thm:b}
\item For all integers $k\ge4$ and $\ell=2$ and if $p=\frac{\omega}{n^{k/2}}$,
then a.a.s.\ $H_{n,p}^{(k)}$ is $2$-offset Hamiltonian.\label{thm:c}
\end{enumerate}
\end{theorem}
\noindent
\noindent
The proof for Theorem~\ref{thm:main} presented in the next section is based on the second moment method, similar to the proof of Theorem~\ref{thm:tight}. Observe that the only case not covered by Theorem~\ref{thm:main} is $\ell=1$; we will comment on this in Section~\ref{sec:conl}. Moreover, we will also see in Sections~\ref{sec:123} and~\ref{sec:connect} that the structures captured by offset Hamiltonian cycles arise in a natural way in a problem related to the 1-2-3 Conjecture.

\section{Proof of Theorem~\ref{thm:main}}\label{sec:main}

Let $X$ be the random variable that counts the number of $\ell$-offset Hamiltonian cycles. Observe that the number of cycles in the complete $k$-uniform hypergraph is:
\[
\gamma_n:=\frac{n!}{2n} \cdot \frac{k-\ell}{(\ell!(k-\ell)!)^{n/k}}.
\]
Indeed, first we order all vertices into a cycle, next we add $2n/k$ edges, which can be shifted by at most $k-\ell$ positions, and finally we need to correct this by permuting all vertices in any two consecutive edges.

Using Stirling's formula we have
\[
\E(X) =\gamma_n \cdot p^{2n/k} = (1+o(1))(k-\ell)\sqrt{\frac{\pi}{2n}} \cdot \left( \frac{n}{e} \cdot \left(\frac{p^2}{\ell!(k-\ell)!}\right)^{1/k}\right)^n
\]
and letting
\[
p = (1-\varepsilon)\sqrt{e^{k}\ell!(k-\ell)!/n^{k}},
\]
we have
\begin{align*}
\E(X) &= (1+o(1))(k-\ell)\sqrt{\frac{\pi}{2n}} \cdot \left( \frac{n}{e} \cdot \left(\frac{\left((1-\varepsilon)\sqrt{e^{k}\ell!(k-\ell)!/n^{k}}\right)^2}{\ell!(k-\ell)!}\right)^{1/k}\right)^n\\&= (1+o(1))(k-\ell)\sqrt{\frac{\pi}{2n}} (1-\varepsilon)^{2n/k}=o(1).
\end{align*}
This verifies part \eqref{thm:a}.

Now we let 
\[
p = (1+\varepsilon)\sqrt{e^{k}\ell!(k-\ell)!/n^{k}}
\]
and let $H$ be a fixed $\ell$-offset Hamiltonian cycle. Observe that 
\[
\E(X)=(1+o(1))(k-\ell)\sqrt{\frac{\pi}{2n}} (1+\varepsilon)^{2n/k}=\infty.
\]
Let $N(b,a)$ be the number of $H'$ $\ell$-offset Hamiltonian cycles such that $|E(H)\cap E(H')|=b$ and $E(H)\cap E(H')$ consists of $a$ edge disjoint paths. Since trivially $N(0,0)\le \gamma_n$, we obtain
\begin{align*}\label{eq:var0}
\begin{split}
\frac{\E(X^2)}{\E(X)^2}&=\frac{\gamma_n N(0,0)p^{4n/k}}{\E(X)^2}
+ \sum_{b=1}^{2n/k} \sum_{a=1}^{\min\{b,n/k\}}\frac{\gamma_n
N(b,a) p^{4n/k-b}} {\E(X)^2}\\
&\le 1 + \sum_{b=1}^{2n/k} \sum_{a=1}^{\min\{b,n/k\}}\frac{\gamma_n
N(b,a) p^{2n/k-b}} {\E(X)}.
\end{split}
\end{align*}

It remains to show that 
\[
\sum_{b} \sum_{a} \frac{N(b,a)
p^{2n/k-b}} {\E(X)}=o(1)
\]
so that we can use Chebyshev's inequality to imply that
\begin{equation}\label{eq:cheb}
\Pr(X=0) \le \frac{\E(X^2)}{\E(X)^2} - 1=o(1),
\end{equation}
as required.

To find an upper bound on $N(b,a)$ we first consider how many ways we can find paths $P_1,P_2,\ldots,P_a$ with a total of $b$ edges. To begin, for each $1\le i\le a$ choose vertices $v_i$ on $V(H)$. We have at most
\begin{equation}\label{x2}
n^a
\end{equation}
choices. Let
\[
b_1+b_2+\dots+b_a = b,
\]
where $b_i\ge 1$ is an integer for every $1\le i\le a$. Note that this equation has exactly
\begin{equation}\label{x3}
{b-1\choose a-1}
\end{equation}
solutions. So for every $i$, we choose a path of length $b_i$ in $H$ which starts at $v_i$ and it moves clockwise. Thus we \eqref{x2} and \eqref{x3} tell us we have at most
\begin{equation}\label{eq:b}
{b-1\choose a-1} n^a
\end{equation}
ways to choose our paths.

Now we count the number of $H'$ containing $P_1,\dots,P_a$. For each even path (that means with even number of edges) 
\[
|V(P_i)| = \frac{b_ik}{2}+\ell \quad \text{ or }\quad |V(P_i)| = \frac{b_ik}{2}+(k-\ell)
\]
and for each odd path
\[
|V(P_i)| = \frac{b_ik}{2}+\frac{k}{2}.
\]
Since $\ell \le k/2$ then for all paths we have
\[
|V(P_i)| \ge \frac{b_ik}{2}+\ell.
\]
Then
\begin{align*}
\sum_{i} |V(P_i)| &\ge \sum_{i} \left( \frac{b_ik}{2}+\ell \right)=bk/2+a\ell.
\end{align*}
Thus, we have at most $n - bk/2-a\ell$ vertices not in $\bigcup_{i=1}^aV(P_i)$. Observe that $H'$ is uniquely determined by the sequence of $2n/k$ subsets each of sizes alternating from $k-\ell$ to $\ell$. For each $V(P_i)$, if $b_i=1$ then we need to divide $|V(P_i)|=k$ vertices into 2 subsets of size $k-\ell$ and $\ell$. The number of ways these paths can be split into alternating subsets is at most
\begin{equation}\label{eq:2ka}
\binom{k}{\ell}^a\le\binom{k}{k/2}^a < 2^{ka}.
\end{equation}
(If $b_i>1$, then there is nothing to do.)

Next we divide the vertices in $V(H) \setminus (V(P_1)\cup\dots\cup V(P_a))$ into subsets of size $\ell$ and $k-\ell$ to obtaining a cycle of alternating subsets. Let $b_i'$ be the number of edges in $H'$ that lie between $P_i$ and $P_{i+1}$ and connect $P_i$ with $P_{i+1}$.
Then there are exactly $b_i'-1$ alternating subsets between $P_i$ and $P_{i+1}$ in $H'$. Thus, we have at least $(b_i'-2)/2$ groups of size $\ell$ and of size $k-\ell$ between $P_i$ and $P_{i+1}$. Since
\[
\sum_{i=1}^a (b_i'-2)/2=\left(\frac{2n}{k}-b\right)/2-a = n/k - b/2  -a,
\]
we conclude that we have at least $(n/k - b/2  -a)$ groups of size $\ell$ and at least $(n/k - b/2  -a)$ groups of size $k-\ell$ on $V(H) \setminus (V(P_1)\cup\dots\cup V(P_a))$. Consequently, we can divide $V(H) \setminus (V(P_1)\cup\dots\cup V(P_a))$ into alternating groups in at most
\begin{equation}\label{x4}
\frac{(n-bk/2-a\ell)!}{2(n-bk/2-a\ell)} \cdot \frac{1}{(\ell!(k-\ell)!)^{n/k - b/2  -a}} = \left(n-bk/2-a\ell-1\right)! \cdot\frac{1}{2(\ell!(k-\ell)!)^{n/k - b/2  -a}}
\end{equation}
choices.

Now mark $a$ positions to insert $P_i$'s. We can trivially do it in 
\begin{equation}\label{x5}
(n-bk/2-a\ell)^a \le (n-bk/2-a\ell) \cdot n^{a-1}
\end{equation}
ways. 

Using \eqref{eq:2ka}, \eqref{x4} and \eqref{x5}, the number $H'$'s containing $P_1, P_2,\dots,P_a$ is smaller than
\begin{equation}\label{eq:pi}
2^{ka} \cdot (n-bk/2-a\ell)! \frac{1}{(\ell!(k-\ell)!)^{n/k - b/2  -a}} \cdot n^{a-1}.
\end{equation}
Thus, by \eqref{eq:b} and \eqref{eq:pi} we obtain
\[
N(b,a) < {b-1\choose a-1} \cdot 2^{ka} \cdot (n-bk/2-a\ell)! \frac{1}{(\ell!(k-\ell)!)^{n/k - b/2  -a}} \cdot n^{2a-1}
\]
and so
\begin{align*}
\frac{N(b,a) p^{2n/k-b}} {\E(X)} &\le {b-1\choose a-1} \frac{2^{ka}\cdot(n-bk/2-a\ell)!\cdot n^{2a-1}\cdot p^{2n/k-b}\cdot 2n\cdot(\ell!(k-\ell)!)^{n/k}}{(\ell!(k-\ell)!)^{n/k - b/2  -a}\cdot n!\cdot (k-\ell)p^{2n/k}}\\
&= {b-1\choose a-1} \frac{(\ell!(k-\ell)!)^{b/2}\cdot(n-bk/2-a\ell)!}{n!\cdot p^b}\cdot \frac{(\ell!(k-\ell)!)^{a}\cdot 2^{ka+1}\cdot n^{2a}}{(k-\ell)}.
\end{align*}
Using Stirling's approximation, letting $p=(1+\varepsilon)\sqrt{e^{k}\ell!(k-\ell)!/n^{k}}$, and observing that $n-bk/2-a\ell\le n$ we have
\begin{align*}
\frac{N(b,a) p^{2n/k-b}} {\E(X)}&\le {b-1\choose a-1} \frac{(\ell!(k-\ell)!)^{b/2}\cdot\left(\frac{n-bk/2-a\ell}{e}\right)^{(n-bk/2-a\ell)}}{\left(\frac{n}{e}\right)^{n}\cdot(1+\varepsilon)^b \cdot \frac{e^{bk/2}(\ell!(k-\ell)!)^{b/2}}{n^{bk/2}}}\cdot \frac{(\ell!(k-\ell)!)^{a}\cdot 2^{ka+1}\cdot n^{2a}}{(k-\ell)}\\
&\le {b-1\choose a-1} \frac{(\ell!(k-\ell)!)^{b/2}\cdot\left(\frac{n}{e}\right)^{(n-bk/2-a\ell)}}{\left(\frac{n}{e}\right)^{n}\cdot(1+\varepsilon)^b \cdot \frac{e^{bk/2}(\ell!(k-\ell)!)^{b/2}}{n^{bk/2}}}\cdot \frac{(\ell!(k-\ell)!)^{a}\cdot 2^{ka+1}\cdot n^{2a}}{(k-\ell)}\\
&=\frac{2}{k-\ell}\left(\frac{1}{1+\varepsilon}\right)^b{b-1\choose a-1}\left(\frac{2^k(e^{\ell}\ell!(k-\ell)!}{n^{\ell-2}}\right)^a.
\end{align*}
This implies that
\begin{align*}
\sum_{b}\sum_{a}\frac{N(b,a) p^{2n/k-b}} {\E(X)}&\le\frac{2}{k-\ell}\sum_{b}\left(\frac{1}{1+\varepsilon}\right)^b\sum_{a}{b-1\choose a-1}\left(\frac{2^ke^{\ell}\ell!(k-\ell)!}{n^{\ell-2}}\right)^a.
\end{align*}
Since
\begin{align*}
\sum_{a=1}^b{b-1\choose a-1}\left(\frac{2^ke^{\ell}\ell!(k-\ell)!}{n^{\ell-2}}\right)^a
&=\frac{2^ke^{\ell}\ell!(k-\ell)!}{n^{\ell-2}}\sum_{a=1}^{b}{b-1\choose a-1}\left(\frac{2^ke^{\ell}\ell!(k-\ell)!}{n^{\ell-2}}\right)^{a-1}\\
&=\frac{2^ke^{\ell}\ell!(k-\ell)!}{n^{\ell-2}}\left(1+\frac{2^ke^{\ell}\ell!(k-\ell)!}{n^{\ell-2}}\right)^{b-1}\\
&\le\frac{O(1)}{n^{\ell-2}}\left(1+\frac{O(1)}{n^{\ell-2}}\right)^b,
\end{align*}
we get for $\ell\ge 3$ and $\varepsilon >0$ that 
\[
\sum_{b} \sum_{a} \frac{N(b,a) p^{2n/k-b}} {\E(X)}\le\frac{O(1)}{n^{\ell-2}}  \sum_b \left(\frac{1 +   \frac{O(1)}{n^{\ell-2}}  }{1+\varepsilon}\right)^b
\le \frac{O(1)}{n^{\ell-2}} \cdot O(\varepsilon) = o(1).
\]
This proves part \eqref{thm:b}.

Furthermore, if $\ell=2$ and $\displaystyle{p=\frac{\omega}{n^{k/2}}}$ (that means $\varepsilon=\omega$), then
\[
\sum_{b} \sum_{a} \frac{N(b,a) p^{2n/k-b}} {\E(X)}
\le O(1) \sum_{b\ge 1}   \left(\frac{1 +   O(1)  }{\omega}\right)^b
\le O(1) \cdot \frac{1 +   O(1)  }{\omega} = o(1).
\]
This proves part \eqref{thm:c} and completes the proof of Theorem~\ref{thm:main}.

\section{Group colorings}\label{sec:123}

The well-known 1-2-3 Conjecture of Karo\'nski, \L{}uczak and Thomason~\cite{KLT} asserts that in every graph (without isolated edges) 
the edges can have assigned weights from $\{1,2,3\}$ so that adjacent vertices have different sums of incident edge weights.
This conjecture attracted a lot of attention and has been studied by several researchers (see, e.g.,  a survey paper of Seamone~\cite{Seamone}). 
The 1-2-3 conjecture is still open but Kalkowski, Karo\'nski, and Pfender~\cite{5enough} have shown that the conjecture holds if $\{1,2,3\}$ is replaced by $\{1,\dots,5\}$. (For previous results see  \cite{30enough, Gnp, 13enough}).

One can extend these ideas by considering $k$-uniform hypergraphs $H$. A vertex coloring of $H$ is \emph{weak} if $H$ has no monochromatic edge and \emph{strong} if for each edge all vertices within that edge have distinct colors. Then we say that $H$ is \emph{weakly} \emph{$w$-weighted} if there exists an edge coloring from $[w]$ induces a weak vertex-coloring. Similarly, we say that $H$ is \emph{strongly} \emph{$w$-weighted} if the corresponding coloring is strong.
Clearly each strongly $w$-weighted hypergraph is also weakly $w$-weighted. Note that for graphs ($k=2$) weak and strong colorings (and therefore weightings) are the same.

In~\cite{KKP} Kalkowski, Karo\'nski, and Pfender studied weakly weighted hypergraphs.  In particular, they proved that any $k$-uniform hypergraph without isolated edges is weakly $(5k-5)$-weighted and that 3-uniform hypergraphs are even weakly 5-weighted. They also asked whether there is an absolute constant~$w_0$ such that every $k$-uniform hypergraph is weakly $w_0$-weighted. Furthermore, they conjectured that each 3-uniform hypergraph without isolated edges is weakly 3-weighted. It was shown by Bennett, Dudek, Frieze and Helenius~\cite{BDFH} that for almost all uniform hypergraphs these conjectures hold. 

In this paper we explore another direction, where edge-weights are elements of an abelian group. This was introduced by Karo\'nski, \L{}uczak and Thomason~\cite{KLT}.
\begin{theorem}[\cite{KLT}]\label{thm:Gabel}
Let $\Gamma$ be a finite abelian group of odd order and let $G$ be a non-trivial $|\Gamma|$-colorable graph. Then there is a weighting of the edges of $G$ with the elements of $\Gamma$ such that the resultant vertex weighting is a proper coloring.
\end{theorem}
\noindent It was our attempts to extend this idea to $k$-uniform hypergraphs that led us to consider the concept of offset Hamilton cycles. So for $u,v\in V(H)$ let $T$ be a $uv$-trail (that means a sequence of vertices with repeated vertices allowed) in $H$ such that
\begin{enumerate}[(i)]
\item the first two edges have $k-1$ vertices in common not including $u$, and 
\item the last two edges have $k-1$ vertices in common not including $v$, and 
\item any two successive edges in the trail have either $1$ or $k-1$ vertices in common in an alternating fashion. 
\end{enumerate}

Any trail that fits this pattern will be called \emph{$1$-offset}. Thus the vertices along $T$ are subdivided into alternating groups of size $1$ and $k-1$, starting and ending with subdivisions of $1$. Observe that this condition implies the number of edges in $T$ must be even. Also, since $T$ is a trail, vertices can get used in multiple edges, including $u$ and $v$, but the first two edges must start with $u$ as a singleton and the last two edges must finish with $v$ as a singleton~(see Figure~\ref{fig:trail}). Let $\mathscr{T}$ be the hypergraph property that for all pairs of vertices $u,v\in V(H)$ there exists a $1$-offset $uv$-trail. If hypergraph $H\in \mathscr{T}$, then we say that $H$ is $\mathscr{T}$-\emph{connected}.

\begin{figure}[h]
\begin{tikzpicture}
\draw [white, fill=gray!70!white, rounded corners=12.5pt] 
    (-175:2.45) arc (-175:-240:2.45) -- (3.25,2.2) -- (3.25,1.3) --
    (-240:1.55) arc (-240:-175:1.55) -- cycle;
\draw [white, fill=gray!50!white, rounded corners=10pt] 
    (-115:2.35) arc (-115:-240:2.35) -- (1.25,2.1) -- (1.25,1.4) --
    (-240:1.65) arc (-240:-115:1.65) -- cycle;
\draw [white, fill=gray!30!white, rounded corners=7.5pt] 
    (65:2.25) arc (65:-125:2.25) -- 
    (-125:1.75) arc (-125:65:1.75) -- cycle;
\draw [white, fill=gray!10!white, rounded corners] 
    (122:2.3)--(65:2.15) arc (65:-65:2.15) -- 
    (-65:1.85) arc (-65:60:1.85) -- (127:2)-- cycle;
\draw [rounded corners=12.5pt] 
    (-175:2.45) arc (-175:-240:2.45) -- (3.25,2.2) -- (3.25,1.3) --
    (-240:1.55) arc (-240:-175:1.55) -- cycle;
\draw [rounded corners=10pt] 
    (-115:2.35) arc (-115:-240:2.35) -- (1.25,2.1) -- (1.25,1.4) --
    (-240:1.65) arc (-240:-115:1.65) -- cycle;
\draw [rounded corners=7.5pt] 
    (65:2.25) arc (65:-125:2.25) -- 
    (-125:1.75) arc (-125:65:1.75) -- cycle;
\draw [rounded corners] 
    (122:2.3)--(65:2.15) arc (65:-65:2.15) -- 
    (-65:1.85) arc (-65:60:1.85) -- (127:2)-- cycle;
\foreach \n in {1,...,6}{\draw[fill=black] (\n*60:2) circle (0.05);};
\draw[fill=black] (30:3.464) circle (0.05);
\node [left] at (125:2) {$u$};
\node [left] at (30:3.464) {$v$};
\node at (172:2) {$v_5$};
\node [left] at (236:1.85) {$v_4$};
\node [left] at (296:2.15) {$v_3$};
\node at (350:2.05) {$v_2$};
\node [right] at (415:2.3) {$v_1$};
\end{tikzpicture}
\caption{A $1$-offset $uv$-trail, $u,v_1,v_2,v_3,v_4,v_5,u,v_1,v$, in a $4$-uniform hypergraph that consists of $4$ edges.}
\label{fig:trail}
\end{figure}
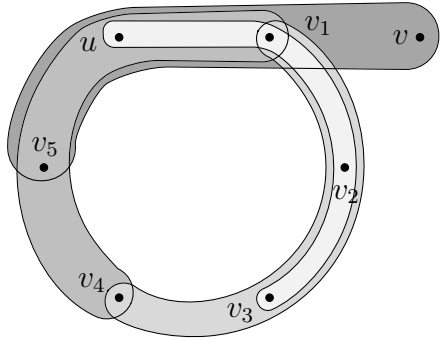

Hypergraph property $\mathscr{T}$ is what allows us to state and prove a result analogous to Theorem~\ref{thm:Gabel} with only slight modification of the proof as presented in~\cite{KLT}.
\begin{theorem}\label{thm:Habel}
Let $\Gamma$ be a finite abelian group of order $w$ and let $H$ be a $k$-uniform hypergraph that is $\mathscr{T}$-connected and  strongly (weakly) $w$-colorable. Furthermore, let $\gcd{(w,k)}=1$. Then $H$ is strongly (weakly) $w$-weighted by the elements of $\Gamma$.
\end{theorem}
We will need a simple fact.
\begin{observation}\label{claim:abel}
Let $\Gamma$ be an additive finite abelian group of order $w$ and let $k$ be a positive integer with $\gcd{(w,k)}=1$. Then for every $g\in \Gamma$
there exists $h\in\Gamma$ such that $g=kh$.
\end{observation}
\begin{proof}
Define the group homomorphism $\varphi : \Gamma\to\Gamma$ as $\varphi(g)=kg$. Then $\operatorname{ker}\varphi$ consists of the identity and all group elements of order $k$. But $\gcd{(w,k)}=1$, so there are no groups elements of order $k$ by Lagrange's Theorem. Thus $\operatorname{ker}\varphi$ is trivial and since $\Gamma$ is finite, we conclude that $\varphi$ is an isomorphism. Thus for any $g\in\Gamma$ there exists $h\in\Gamma$ such that $g=kh$.
\end{proof}

\begin{proof}[Proof of Theorem~\ref{thm:Habel}]
Fix a strong (weak) vertex coloring $c: V\to \Gamma$ of $H=(V,\mathcal{E})$. Then by Observation~\ref{claim:abel}, we know that there exists $h\in\Gamma$ such that 
\[
\sum_{v\in V} c(v)=kh.
\]
Now select an arbitrary $f\in \mathcal{E}$ and let it have weight $h$ with all other edges given weight $0$. This induces a vertex coloring $c'$  and if for all vertices $v$ we have $c(v) = c'(v)$, then there is nothing else to do. So we may assume that there exists a vertex $x$ such that $c(x)\ne c'(x)$. Then there must be another vertex $y\neq x$ such that $c(y)\ne c'(y)$. This immediately follows from the fact that $\sum_{v} c(v) = \sum_{v} c'(v)$.

Then let $T$ be the trail between $x$ and $y$ guaranteed to us by property $\mathscr{T}$ and let $g=c'(x)-c(x)$. By alternately subtracting and adding $g$ along each edge of $T$, we redefine $c'$ and end up with $c(x)=c'(x)$. Furthermore, since $T$ has an even length, the colors induced by $c'$ of all the other vertices of $T$ remain unchanged except possibly $y$. More importantly,  equality $\sum_{v} c(v) = \sum_{v} c'(v)$ still holds.

Repeated application of this process will eventually terminate in an edge weighting $c'$ of $H$ for which $c(v)=c'(v)$ for all $v\in V$. This is because once a vertex has been corrected, it can only ever be an internal vertex for every future trail chosen, which leaves the corrected coloring unchanged for all future iterations and each iteration of this process leaves us with at least one less vertex with an incorrect induced color. 
\end{proof}

\section{$\mathscr{T}$-Connectivity}\label{sec:connect}

In the previous section, we assume that $H$ is $\mathscr{T}$-connected. We would like to know what the threshold for $\mathscr{T}$-connectivity in $H_{n,p}^{(k)}$ might be. Heuristically, since the probability that a given $1$-offset $uv$-trail of length $2$ existing is $p^2$, then for fixed $u$ and $v$ the probability that \emph{no} $1$-offset $uv$-trail of length $2$ existing is exactly 
\[
p'=(1-p^2)^{\binom{n-2}{k-1}}.
\]
If we assume that the choices for each pair of vertices are independent, then we should be able to model the threshold for $\mathscr{T}$-connectivity in $H_{n,p}^{(k)}$ with connectivity in $G_{n,1-p'}$, which we know to be connected when $1-p'=(\log{n} + \omega)/n$. So
\[
1-p'=1-(1-p^2)^{\binom{n-2}{k-1}}\approx p^2\frac{n^{k-1}}{(k-1)!}
\]
and that leaves us with probability $\sqrt{(k-1)!(\log{n})/n^k}$ as a target for the sharp threshold for $\mathscr{T}$-connectivity in $H_{n,p}^{(k)}$. This heuristic argument turns out to be accurate and $\sqrt{(k-1)!(\log{n})/n^k}$ is the sharp threshold for $\mathscr{T}$-connectivity.
\begin{theorem}
Let $k\ge 3$ be an integer and $\varepsilon >0$. Then:
\begin{enumerate}[(i)]
\item If $p \ge (1+\varepsilon) \sqrt{(k-1)!(\log{n})/n^k}$, then a.a.s. $H_{n,p}^{(k)}\in\mathscr{T}$.\label{thm:connect:a}
\item If $p \le (1-\varepsilon) \sqrt{(k-1)!(\log{n})/n^k}$, then a.a.s. $H_{n,p}^{(k)}\notin\mathscr{T}$.\label{thm:connect:b}
\end{enumerate}
\end{theorem}
\begin{proof}
First we show part~\eqref{thm:connect:a}. Let us divide $V$ into two sets, $S$ and $V\setminus S$ such that $\lvert S\rvert=s$ with $1\le s\le n/2$ and let $X_S$ be the number of $1$-offset trails of length $2$ connecting the two sets. Since the trails must start in one set and end in another there should be $\ell = s(n-s)\binom{n-2}{k-1}$ potential trails. Now suppose we enumerate each of the trails and let $X_i$ be the indicator variable that the $i$th trail is present. Clearly,  $X_S=X_1+X_2+\cdots+X_\ell$ and
\[
\mu=\E(X_S)=s(n-s)\binom{n-2}{k-1}p^2 = (1+o(1))(1+\varepsilon)^2s\frac{n-s}{n}\log n.
\]
Observe that some of these trails share edges with the others, so the event that $X_i=1$ and $X_j=1$  is not necessarily independent for all $i$ and $j$. So we write $i\sim j$ if $X_i$ and $X_j$ share an edge, then
\[
\Delta = \sum_{\substack{{\{i,j\}:i\sim j} \\ {i\ne j}}}\E(X_iX_j) = O\left( s(n-s)\binom{n-2}{k-1} n p^3\right) = O\left(s (\log n)^{3/2} / n^{k/2 - 1} \right).
\]
So by Janson's inequality (see, e.g., Corollary 21.13 in \cite{FK}) we have
\[
\Pr(X_S=0)\le e^{-\mu+\Delta}\le e^{-(1+o(1))(1+\varepsilon)^2s\frac{n-s}{n}\log n}
\]
and the union bound taken over all sets $S$ of size $1\le s \le n/2$ implies
\begin{align*}
\sum_{s=1}^{n/2} \binom{n}{s} \Pr(X_S=0)  &= \sum_{s=1}^{1/\varepsilon} \binom{n}{s} \Pr(X_S=0) \\
&\qquad +   \sum_{s=1+ 1/\varepsilon}^{n/\log n} \binom{n}{s} \Pr(X_S=0) + \sum_{s=1+n/\log n}^{n/2} \binom{n}{s} \Pr(X_S=0).
\end{align*}
If $1\le s \le 1/\varepsilon$, then $(n-s)/n \approx 1$ and 
\[
\sum_{s=1}^{1/\varepsilon} \binom{n}{s} \Pr(X_S=0) \le \sum_{s=1}^{1/\varepsilon} n^s e^{-(1+o(1))(1+\varepsilon)^2s\log n}
\le \sum_{s=1}^{1/\varepsilon} \frac{1}{n^{\varepsilon s}} \le \frac{1}{\varepsilon} \cdot \frac{1}{n^{\varepsilon}} = o(1). 
\]
If $1/\varepsilon \le s \le n/\log n$, then again $(n-s)/n  \approx 1$ and 
\[
\sum_{s=1+1/\varepsilon}^{n/\log n} \binom{n}{s} \Pr(X_S=0) \le \sum_{s=1+1/\varepsilon}^{n/\log n} n^s e^{-(1+o(1))(1+\varepsilon)^2s\log n} 
\le \sum_{s=1+1/\varepsilon}^{n/\log n} \frac{1}{n^{\varepsilon s}} 
\le \frac{n}{\log n} \cdot \frac{1}{n} = o(1).
\]
Finally, if $n/\log n \le s \le n/2$, then $(n-s)/n \ge 1/2$ and since $\binom{n}{s} \le \left( en/s\right)^s \le (e\log n)^s$, we get
\begin{align*}
\sum_{s=1+n/\log n}^{n/2} \binom{n}{s} \Pr(X_S=0) &\le \sum_{s=1+n/\log n}^{n/2} (e\log n)^s e^{-(1+o(1))(1+\varepsilon)^2s(\log n)/2}\\
&=\sum_{s=1+n/\log n}^{n/2} e^{s + s\log\log n-(1+o(1))(1+\varepsilon)^2s(\log n)/2} = o(1).
\end{align*}
This completes the proof of part~\eqref{thm:connect:a}. 

Now we prove part~\eqref{thm:connect:b}. Clearly it suffices to show that if $p = (1-\varepsilon) \sqrt{(k-1)!(\log{n})/n^k}$, then $H_{n,p}^{(k)}$ has a vertex which is not an endpoint of any $1$-offset trail of length $2$. For a fixed vertex $v$ let $X_v$ counts the number of $1$-offset trails of length $2$ with $v$ as an endpoint. Let $Y_v$ be an indicator random variable which equal to 1 if $X_v=0$. Let $Y=\sum_v Y_v$. Since there are $\binom{n-1}{k} \binom{k}{k-1}$ potential trails with $v$ as its endpoint, 
the FKG inequality (see, e.g., Theorem 21.5 in \cite{FK}) implies that 
\[
\Pr(Y = 1) = \Pr(X_v = 0) \ge (1-p^2)^{\binom{n-1}{k} \binom{k}{k-1}} 
\]
and so
\begin{equation}\label{eq:y}
\E(Y) \ge n (1-p^2)^{\binom{n-1}{k} \binom{k}{k-1}}  \approx ne^{-\binom{n-1}{k} \binom{k}{k-1}p^2} \to \infty.
\end{equation}

Let $v \neq w$. Now we will use Janson's inequality to estimate from above
\[
\Pr(Y_v = Y_w = 1) = \Pr(X_v = X_w = 0) = \Pr(X_v + X_w = 0).
\]
Let 
\[
X_v + X_w = \sum_{i} Z_i^v + \sum_{i} Z_i^w,
\] 
where $Z_i^v$ and $Z_i^w$ are potential 1-offset trails of length 2 with an endpoint $v$ or $w$, respectively. Clearly,
\[
\E(X_v + X_w) = 2\binom{n-1}{k} \binom{k}{k-1} p^2
\]
and
\[
\Delta = \sum_{\substack{{\{i,j\}:i\sim j} \\ {i\ne j}}}\E(Z_i^vZ_j^v) + \sum_{\substack{{\{i,j\}:i\sim j} \\ {i\ne j}}}\E(Z_i^wZ_j^w) + \sum_{\substack{{\{i,j\}:i\sim j} \\ {i\ne j}}}\E(Z_i^vZ_j^w).
\]
Now each of these three sums is bounded by $O(n^{k+1} p^3 + n^{k-2}p^2)$. The first term counts all pairs of 1-offset trials of length 2 with exactly one edge in common. The second term counts 1-offset trials of length 2 with $v$ and $w$ as its endpoints. (Observe that $Z_i^v$ and $Z_i^w$ can be associated with the same trail.) Consequently, $\Delta = o(1)$ and Janson's inequality yields that
\[
\Pr(X_v + X_w = 0) \le e^{-2\binom{n-1}{k} \binom{k}{k-1} p^2 +o(1)}
\]
so that 
\begin{align*}
\E(Y^2) &=  \E(Y) + \sum_{v\neq w} \Pr(Y_v = Y_w = 1)\\
&= \E(Y) + \sum_{v\neq w} \Pr(X_v + X_w = 0)
\le \E(Y) + n^2 e^{-2\binom{n-1}{k} \binom{k}{k-1} p^2 +o(1)}.
\end{align*}
Thus, due to~\eqref{eq:y}, we have $\frac{\E(Y^2)}{\E(Y)^2}  \approx 1$
and then Chebyshev's inequality (cf.~\eqref{eq:cheb}) implies that a.a.s. $Y>0$, as required.
\end{proof}

\section{Concluding remarks}\label{sec:conl}

In this paper we have studied the $\ell$-offset Hamiltonicity of random hypergraphs and why these structures are important in relation to the 1-2-3 Conjecture. The case when $\ell=1$ is not understood yet, but we conjecture that the asymptotic threshold for the existence of $1$-offset Hamilton cycle in $H_{n,p}^{(k)}$ is $\sqrt{(\log n)/n^k}$. In order to prove this, one can try to use a similar approach as in~\cite{F,DF}. This will require to show that a.a.s. $H_{n,p}^{(k)}$ has a factor of 1-offset trails of length~2, which is similar to a  celebrated result of Johansson, Khan and Vu~\cite{JKV} about factors in hypergraphs. 

One can also consider a directed version of Theorem~\ref{thm:main}. Let $\overrightarrow{H}_{n,p}^{(k)}$ be a \emph{directed random hypergraph}, where every ordered $k$-tuple appears independently with probability~$p$. An immediate consequence of the General Clutter Percolation Theorem of McDiarmid~\cite{MC} implies that all results in parts~\eqref{thm:b} and \eqref{thm:c} of Theorem~\ref{thm:main} also hold for $\overrightarrow{H}_{n,p}^{(k)}$. As a matter of fact an easy modification of the proof of Theorem~\ref{thm:main} yields more accurate results. In particular, one can show that $(e/n)^{k/2}$ is the sharp threshold for the existence of the directed $\ell$-offset Hamilton cycle for $k\ge 6$ and $\ell \ge 3$.


\begin{thebibliography}{9}

\bibitem{30enough}
L.~Addario-Berry, K.~Dalal, C.~McDiarmid, B.~A. Reed, and A.~Thomason,
  \emph{Vertex-colouring edge-weightings}, Combinatorica \textbf{27} (2007),
  no.~1, 1--12.

\bibitem{Gnp}
L.~Addario-Berry, K.~Dalal, and B.~A. Reed, \emph{Degree constrained
  subgraphs}, Discrete Appl. Math. \textbf{156} (2008), no.~7, 1168--1174.

\bibitem{AKS} M.~Ajtai, J.~Koml\'{o}s and E.~Szemer\'{e}di,
{\em The first occurrence of Hamilton cycles in random graphs},
Annals of Discrete Mathematics~\textbf{27} (1985), 173--178.

\bibitem{BDFH} P. Bennett, A. Dudek, A. Frieze and L. Helenius,
{\em Weak and strong versions of the 1-2-3 conjecture for uniform hypergraphs},
Electronic Journal of Combinatorics~\textbf{23} (2016), P2.46.

\bibitem{Boll} B.~Bollob\'{a}s,
{\em The evolution of sparse graphs},
in Graph Theory and Combinatorics, Academic Press, Proceedings of Cambridge Combinatorics,
Conference in Honour of Paul Erd\H{o}s (B.~Bollob\'{a}s; Ed) (1984), 35--57.

\bibitem{DF} A.~Dudek and A.~Frieze,
{\em Loose Hamilton cycles in random uniform hypergraphs},
Electronic Journal of Combinatorics~\textbf{18} (2011), P48.

\bibitem{DF2} A.~Dudek and A.~Frieze,
{\em Tight Hamilton cycles in random uniform hypergraphs},
Random Structures and Algorithms~\textbf{42} (2013), no. 3, 374--385. 

\bibitem{DFLS}
A.~Dudek, A.~Frieze, P.-S. Loh, and S.~Speiss,
{\em Optimal divisibility conditions for loose {H}amilton cycles in random hypergraphs},
Electronic Journal of Combinatorics~\textbf{19} (2012), P44.

\bibitem{Feb} A.~Ferber,
{\em Closing gaps in problems related to Hamilton cycles in random graphs and hypergraphs},
Electronic Journal of Combinatorics~\textbf{22} (2015), P1.61.

\bibitem{F} A.~Frieze,
{\em Loose Hamilton cycles in random 3-uniform hypergraphs},
Electronic Journal of Combinatorics~\textbf{17} (2010), N28.
  
\bibitem{FK} A.~Frieze and M. Karo\'nski,
{\em Introduction To Random Graphs},
Cambridge University Press, 2016.  
  
\bibitem{JKV} A.~Johansson, J.~Kahn and V.~Vu,
{\em Factors in random graphs},
Random Structures and Algorithms \textbf{33} (2008), 1--28.  
  
\bibitem{5enough}
M.~Kalkowski, M.~Karo\'nski, and F.~Pfender, \emph{Vertex-coloring
  edge-weightings: Towards the 1-2-3-conjecture}, J. Comb. Theory, Ser. B
  \textbf{100} (2010), no.~3, 347--349.  
  
\bibitem{KKP}
\bysame, \emph{The 1-2-3 conjecture for hypergraphs}, E-print: \texttt{arXiv:1308.0611},
  2013.

\bibitem{KLT} M. Karo\'nski, T. {\L}uczak, A. Thomason,
{\em Edge weights and vertex colors},
Journal of Combinatorial Theory, Ser. B \textbf{91} (2004), 151--157.

\bibitem{MC} C.~McDiamird,
{\em General percolation and random graphs},
Advances in Applied Probability~\textbf{13} (1981), 40--60.

\bibitem{KS} J.~Koml\'os and E.~Szemer\'edi,
{\em Limit distributions for the existence of Hamilton circuits in a random graph},
Discrete Mathematics~\textbf{43} (1983), 55--63.
     
\bibitem{PP} O.~Parczyk and Y.~Person,
{\em Spanning structures and universality in sparse hypergraphs},
Random Structures and Algorithms~\textbf{49} (2016), no. 4, 819--844. 
 
\bibitem{Poole}  D. Poole,
{\em  On weak hamiltonicity of a random hypergraph}, 
E-print:  \texttt{arXiv:1410.7446}, 2014.  
  
\bibitem{Seamone}
B.~Seamone, \emph{The 1-2-3 conjecture and related problems: a survey},
  E-print: \texttt{arXiv:1211.5122}, 2012.  
  
\bibitem{13enough}
T.~Wang and Q.~Yu, \emph{On vertex-coloring 13-edge-weighting}, Front. Math.
  China \textbf{3} (2008), no.~4, 581--587.  
  
\end{thebibliography}
\end{document}